\documentclass[12pt]{article}
\usepackage{booktabs}
\usepackage{caption}
\usepackage{amsmath}
\usepackage{amsfonts,amsthm,amssymb,mathrsfs,bbding}
\usepackage{enumerate}
\usepackage{tikz}
\usepackage{rotating}
\usepackage{tabularx}
\allowdisplaybreaks[4]
\evensidemargin=\oddsidemargin\topmargin=-1.5cm

\newtheorem{thm}{Theorem}[section]
\newtheorem{prob}{Problem}

\newtheorem{lem}[thm]{Lemma}

\newtheorem{pro}{Proposition}[section]

\newtheorem{conj}{Conjecture}

\theoremstyle{definition}

\newcommand{\red}[1]{\textcolor{red}{#1}}

\begin{document}
\title{Spectral radius and edge-disjoint spanning trees}
\author{Dandan Fan$^{a,b}$, Xiaofeng Gu$^c$, Huiqiu Lin$^{a,}$\thanks{Corresponding author. {\it E-mail}: {\tt ddfan0526@163.com} (D. Fan), {\tt xgu@westga.edu} (X. Gu), {\tt huiqiulin@126.com} (H. Lin)}\\[2mm]
\small\it $^a$ School of Mathematics, East China University of Science and Technology, \\
\small\it   Shanghai 200237, China\\[1mm]
\small\it $^b$ College of Mathematics and Physics, Xinjiang Agricultural University,\\
\small\it Urumqi, Xinjiang 830052, China\\[1mm]
\small\it $^c$ Department of Computing and Mathematics, University of West Georgia, \\
\small\it Carrollton, GA 30118, USA
}

\date{}
\maketitle
\begin{abstract}
The spanning tree packing number of a graph $G$, denoted by $\tau(G)$, is the maximum number of edge-disjoint spanning trees contained in $G$. The study of $\tau(G)$ is one of the classic problems in graph theory. Cioab\u{a} and Wong initiated to investigate $\tau(G)$ from spectral perspectives in 2012 and since then, $\tau(G)$ has been well studied using the second largest eigenvalue of the adjacency matrix in the past decade. In this paper, we further extend the results in terms of the number of edges and the spectral radius, respectively; and prove tight sufficient conditions to guarantee $\tau(G)\geq k$ with extremal graphs characterized. Moreover, we confirm a conjecture of Ning, Lu and Wang on characterizing graphs with the maximum spectral radius among all graphs with a given order as well as fixed minimum degree and fixed edge connectivity. Our results have important applications in rigidity and nowhere-zero flows. We conclude with some open problems in the end.
\end{abstract}



\section{Introduction}
In this paper, we only consider simply graphs unless otherwise stated and $k$ always denotes a positive integer. The study of edge-disjoint spanning trees has been shown to be very important to graphs and has many applications in fault-tolerance networks as well as network reliability~\cite{Cunningham,Hobbs91}. Thus it is quite interesting to explore how many edge-disjoint spanning trees in a given graph. The \textit{spanning tree packing number} (or simply \textit{STP number}) of a graph $G$, denoted by $\tau(G)$, is the maximum number of edge-disjoint spanning trees contained in $G$. Nash-williams~\cite{Nash-Williams}  and Tutte~\cite{Tutte} independently discovered a fundamental theorem that characterizes graphs $G$ with $\tau(G)\ge k$ (See Theorem~\ref{lem::2.9} in the next section). The \textit{edge connectivity} of $G$, denote by $\kappa'(G)$, is the minimum cardinality of an edge cut of $G$. It is known that $\kappa'(G)$ and $\tau(G)$ are closely related. In fact, the fundamental theorem of Nash-williams~\cite{Nash-Williams} and Tutte~\cite{Tutte} implies that if $\kappa'(G)\ge 2k$, then $\tau(G)\ge k$. For more about the spanning tree packing number, we refer readers to the survey~\cite{Palmer01} by Palmer.


\medskip
The number of spanning trees has also been well studied from spectral perspectives. For a graph $G$, let $A(G)$ denote  the adjacency matrix of $G$ and let $\lambda_i(G)$ denote the $i$th largest eigenvalue of $A(G)$. In particular, the largest eigenvalue of $A(G)$ is called the \textit{spectral radius} of $G$ and is denoted by $\rho(G)$. Denote by $D(G)$ the diagonal matrix of vertex degrees of $G$. The \textit{Laplacian matrix} of $G$ is defined as $L(G)=D(G)-A(G)$. The Laplacian matrix is positive semidefinite and we order its eigenvalues as $\mu_{1}\geq\mu_{2}\geq\ldots\geq\mu_{n-1}\geq\mu_{n}=0$. The well-known Matrix-Tree Theorem of Kirchhoff \cite{Kirchhoff} indicates that the number of spanning trees (not necessarily edge-disjoint) of a graph $G$ with $n$ labelled vertices is $\frac{\prod_{i=1}^{n-1}\mu_{i}}{n}$. 
For edge-disjoint spanning trees, Seymour proposed the following problem in private communication to Cioab\u{a} as mentioned in~\cite{Wong}. 
\begin{prob}\label{prob1}
Let $G$ be a connected graph. Determine the relationship between $\tau(G)$ and the eigenvalues of $G$.
\end{prob}

Inspired by the Kirchhoff's Matrix-Tree Theorem and Problem~\ref{prob1}, Cioab\u{a} and Wong \cite{Wong} started to study the spanning tree packing number via the second largest eigenvalue of the adjacency matrix. They proved that for a $d$-regular connected graph $G$, $\tau(G)\geq k$ if $\lambda_{2}(G)<d-\frac{2(2k-1)}{d+1}$ for $d\geq 2k\geq 4$, and further conjectured that the sufficient condition can be improved to $\lambda_{2}(G)<d-\frac{2k-1}{d+1}$. In the same paper, they did the preliminary work of this conjecture for $k=2,3$ and gave examples to show the bound is best possible. Later, Gu et al.~\cite{Gu-Lai} extended the conjecture to graphs that are not necessarily regular, that is, $\tau(G)\geq k$ if $\lambda_2(G)< \delta-\frac{2k-1}{\delta+1}$ for $\delta\geq 2k\geq 4$. They confirmed the conjecture for $k=2,3$ and obtained a partial result that $\lambda_2(G)< \delta-\frac{3k-1}{\delta+1}$ suffices. This conjecture was completely settled in 2014 by Liu et al.~\cite{Liu-Hong} who proved a stronger result, which also implied the truth of the conjecture of Cioab\u{a} and Wong \cite{Wong}. Most recently, the result in~\cite{Liu-Hong} has been shown to be essentially best possible in \cite{coppww22} by constructing extremal graphs.
On the other hand, the result was extended to a fractional version by Hong et al.~\cite{HGLL16}, and improved by Liu, Lai and Tian \cite{Liu-Lai} with a Moore function. 

Motivated by Problem~\ref{prob1} and the above results, we study the spanning tree packing number by means of the spectral radius of graphs. We first investigate an extremal result for $\tau(G)\ge k$. Let $e(G)$ denote the number of edges in $G$.

\begin{thm}\label{thm::edgenumber}
Let $G$ be a connected graph with minimum degree $\delta\geq 2k$ and order $n\geq 2\delta+2$. If $e(G)\geq {\delta+1\choose2}+{n-\delta-1\choose 2} +k$,
then $\tau(G)\geq k$.
\end{thm}

The condition in Theorem~\ref{thm::edgenumber} is tight. Denote by $K_n$ the complete graph on $n$ vertices, and $\mathcal{G}_{n,n_1}^{i}$ the set of graphs obtained from $K_{n_1}\cup K_{n-n_{1}}$ by adding $i$ edges between $K_{n_1}$ and $K_{n-n_{1}}$. Notice that any graph $G$ in $\mathcal{G}_{n,\delta+1}^{k-1}$ has exactly ${\delta+1\choose2} + {n-\delta-1\choose 2} +k-1$ edges but $\tau(G)<k$.

\medskip
We then focus on a spectral analogue. The corresponding spectral problem is much harder. Let $B_{n,\delta+1}^{i}$ be a graph obtained from $K_{\delta+1}\cup K_{n-\delta-1}$ by adding $i$ edges joining a vertex in $K_{\delta+1}$ and $i$ vertices in $K_{n-\delta-1}$. We succeed in discovering a sufficient condition for $\tau(G)\geq k$ via the spectral radius, and characterize the unique spectral extremal graph $B_{n,\delta+1}^{k-1}$ among the structural extremal graph family $\mathcal{G}_{n,\delta+1}^{k-1}$.
\begin{thm}\label{thm::1.2}
Let $k\geq 2$, and let $G$ be a connected graph with minimum degree $\delta\geq 2k$ and order $n\geq 2\delta+3$. If $\rho(G)\geq \rho(B_{n,\delta+1}^{k-1})$, then $\tau(G)\geq k$ unless $G\cong B_{n,\delta+1}^{k-1}$.
\end{thm}

Our proofs are novel and as an application, we will use similar proof techniques to settle a conjecture of Ning, Lu and Wang \cite{Ning}. Let $\mathcal{A}_{n}^{\kappa',\delta}$ be a set of graphs of order $n$ with minimum degree $\delta$ and edge connectivity $\kappa'$.  For $0\leq \kappa'\leq 3$, Ning, Lu and Wang \cite{Ning} determined that the unique extremal graph with the maximum spectral radius is $B_{n,\delta+1}^{\kappa'}$, and they proposed the following conjecture.

\begin{conj}[Ning, Lu and Wang~\cite{Ning}]\label{conj1}
For $4\leq \kappa'<\delta$, $B_{n,\delta+1}^{\kappa'}$ is the graph with the maximum spectral radius in $\mathcal{A}_{n}^{\kappa',\delta}$.
\end{conj}

We confirm this conjecture for $n\ge 2\delta +4$ as described below.

\begin{thm}\label{thm::1.3}
Let $G\in \mathcal{A}_{n}^{\kappa',\delta}$ where $4\leq\kappa'<\delta$ and $n\geq 2\delta+4$. Then $\rho(G)\leq \rho(B_{n,\delta+1}^{\kappa'})$, with equality if and only if $G\cong B_{n,\delta+1}^{\kappa'}$.
\end{thm}

The proofs of Theorems~\ref{thm::edgenumber} and \ref{thm::1.2} as well as the one of Theorem~\ref{thm::1.3} will be presented in the next two sections, respectively. Since edge-disjoint spanning trees have many applications, we will list some of them in Section~\ref{sec::app}, including rigidity and nowhere-zero flows. Some concluding remarks will be made in Section~\ref{sec::remarks}.

\section{Proofs of Theorems~\ref{thm::edgenumber} and \ref{thm::1.2}}
In this section, we present the proofs of Theorems~\ref{thm::edgenumber} and \ref{thm::1.2}. We first list several lemmas that will be used in the sequel. The following sharp upper bound on the spectral radius was obtained by Hong, Shu and Fang ~\cite{HSF} and Nikiforov~\cite{V.N}, independently.

\begin{lem}[\cite{HSF,V.N}]\label{lem::2.1}
Let $G$ be a graph on $n$ vertices and $m$ edges with minimum degree $\delta\geq 1$. Then
$$\rho(G) \leq \frac{\delta-1}{2}+\sqrt{2 m-n \delta+\frac{(\delta+1)^{2}}{4}},$$
with equality if and only if $G$ is either a $\delta$-regular graph or a bidegreed graph
in which each vertex is of degree either $\delta$ or $n-1$.
\end{lem}

\begin{lem}[\cite{HSF,V.N}]\label{lem::2.2}
For nonnegative integers $p$ and $q$ with $2q \leq p(p-1)$ and $0 \leq x \leq p-1$, the function $f(x)=(x-1) / 2+\sqrt{2 q-p x+(1+x)^{2} / 4}$ is decreasing with respect to $x$.
\end{lem}

Recall that $\mathcal{G}_{n,n_1}^{i}$ is the set of graphs obtained from
$K_{n_1}\cup K_{n-n_{1}}$ by adding $i$ edges between $K_{n_1}$ and $K_{n-n_{1}}$.
\begin{lem}\label{lem::2.3}
Let $k\geq 2$ and let $G\in \mathcal{G}_{n,\delta+1}^{k-1}$ where $n\geq 2\delta+3$ and $\delta\geq 2k$. Then
$$n-\delta-2< \rho(G)< n-\delta-1.$$
\end{lem}

\begin{proof}
Note that $G$ contains $K_{\delta+1}\cup K_{n-\delta-1}$ as a proper spanning subgraph and $n\geq 2\delta+3$. Then $\rho(G)>\rho(K_{\delta+1}\cup K_{n-\delta-1})= n-\delta-2$. Since $\delta\geq 2k\geq 4$, it follows that
\begin{equation*}
\begin{aligned}
   e(G)&={\delta+1\choose 2}+{n-\delta-1\choose 2}+k-1\\
   &\leq {\delta+1\choose 2}+{n-\delta-1\choose 2}+\frac{\delta}{2}-1\\ &=\frac{n^2}{2}-\frac{(2\delta+3)n}{2}+\delta^2+\frac{5\delta}{2}.
\end{aligned}
\end{equation*}
Combining this with Lemmas \ref{lem::2.1} and \ref{lem::2.2}, we have
\begin{equation*}
\begin{aligned}
   \rho(G)&\leq \frac{\delta-1}{2}+\sqrt{n^2 + (-3\delta - 3)n + \frac{9\delta^2}{4}+ \frac{11\delta}{2}+\frac{1}{4}}\\
   &= \frac{\delta-1}{2}+\sqrt{\left(n-\frac{3\delta}{2}-\frac{1}{2}\right)^2-2(n-2\delta)}\\
   &< \frac{\delta-1}{2}+ \left(n-\frac{3\delta}{2}-\frac{1}{2}\right)~~(\mbox{since $n\geq 2\delta\!+\!3$})\\
   &=n-\delta-1.
\end{aligned}
\end{equation*}
This completes the proof.
\end{proof}

It is known that if $G$ is connected, then $A(G)$ is irreducible. By the Perron-Frobenius theorem(cf. \cite[Section 8.8]{C.G-2}), the Perron vector $x$ is a positive eigenvector of $A(G)$ with respect to $\rho(G)$. For any $v\in V(G)$, let $N_{G}(v)$ and $d_G(v)$ be the neighborhood and degree of $v$ in $G$, respectively. Let $N_{G}[v]=N_{G}(v)\cup\{v\}$.

\begin{lem}[\cite{H.L-1}]\label{lem::2.4}
Let $G$ be a connected graph, and let $u,v$ be two vertices of $G$. Suppose that $v_{1},v_{2},\ldots,v_{s}\in N_{G}(v)\backslash N_{G}(u)$ with $s\geq 1$, and $G^*$ is the graph obtained from $G$ by deleting the edges $vv_{i}$ and adding the edges $uv_{i}$ for $1\leq i\leq s$. Let $x$ be the Perron vector of $A(G)$. If $x_{u}\geq x_{v}$, then $\rho(G)<\rho(G^*)$.
\end{lem}

For any vertex $v\in V(G)$ and any subset $S\subseteq V(G)$, let $N_{S}(v)=N_{G}(v)\cap S$ and $d_{S}(v)=|N_{S}(v)|$.

\begin{lem}\label{lem::2.5}
Let $k\geq 2$ and let $G\in \mathcal{G}_{n,\delta+1}^{k-1}$ where $n\geq 2\delta+3$ and $\delta\geq 2k$. Then $\rho(G)\leq \rho(B_{n,\delta+1}^{k-1})$, with equality if and only if $G\cong B_{n,\delta+1}^{k-1}$.
\end{lem}

\begin{proof}
Suppose that $G'$ is the graph that attains the maximum spectral radius in $\mathcal{G}_{n,\delta+1}^{k-1}$. For every $G\in \mathcal{G}_{n,\delta+1}^{k-1}$, we have
\begin{equation}\label{equ::1}
\begin{aligned}
\rho(G)\leq \rho(G').
\end{aligned}
\end{equation}
We partition $V(G')$ into $V_1\cup V_2$ with $V_1=V(K_{\delta+1})=\{u_1,u_2,\ldots,u_{\delta+1}\}$ and $V_2=V(K_{n-\delta-1})=\{v_{1},v_2,\ldots,v_{n-\delta-1}\}$. Let $x$ be the Perron vector of $A(G')$, and let $\rho'=\rho(G')$. Without loss of generality, we assume that $x_{u_{i+1}}\leq x_{u_{i}}$ and $x_{v_{j+1}}\leq x_{v_{j}}$ for $1\leq i\leq \delta$ and $1\leq j\leq n-\delta-2$. We assert that $N_{G'}(u_{i+1})\subseteq N_{G'}[u_i]$ and $N_{G'}(v_{j+1})\subseteq N_{G'}[v_j]$ for $1\leq i\leq \delta$ and $1\leq j\leq n-\delta-2$. If not, suppose that there exist $i,j$ with $i<j$ such that $N_{G'}(u_j)\nsubseteq N_{G'}[u_i]$. Let
$w\in N_{G'}(u_j)\backslash N_{G'}[u_i]$ and $G^{*}=G'-wu_{j}+wu_{i}$. Note that $x_{u_{j}}\leq x_{u_{i}}$. Then $\rho(G^{*})>\rho(G')$ by Lemma \ref{lem::2.4}, which contradicts the maximality of $\rho'$. This implies that  $N_{G'}(u_{i+1})\subseteq N_{G'}[u_i]$ for $1\leq i\leq \delta$. Similarly, we can deduce that $N_{G'}(v_{j+1})\subseteq N_{G'}[v_j]$ for $1\leq j\leq n-\delta-2$. Furthermore, we have
$N_{V_2}(u_{i+1})\subseteq N_{V_2}(u_i)$ and $N_{V_1}(v_{j+1})\subseteq N_{V_1}(v_j)$ for $1\leq i\leq \delta$ and $1\leq j\leq n-\delta-2$.
Let $d_{V_2}(u_1) =r$, $d_{V_2}(u_2)=t$, and $d_{V_1}(v_1)=s$. Again by the maximality of $\rho'$ and Lemma \ref{lem::2.4}, we have $N_{V_2}(u_1)=\{v_1,v_2,\dots,v_r\}$, $N_{V_2}(u_2)=\{v_1,v_2,\dots,v_t\}$ and $N_{V_1}(v_1)=\{u_1,u_2,\dots,u_s\}$. If $r=k-1$ or $s=1$, then $G'\cong B_{n,\delta+1}^{k-1}$. Combining this with (\ref{equ::1}), we can deduce that $\rho(G)\leq\rho(B_{n,\delta+1}^{k-1})$, with  equality if and only if $G\cong B_{n,\delta+1}^{k-1}$, as required. Next we consider $r\leq k-2$ and $s\geq 2$ in the following. Note that $x_{v_i}=x_{v_{r+1}}$ for $r+1\leq i\leq n-\delta-1$, $x_{v_j}\geq x_{v_{r+1}}$ for $2\leq j\leq r$. Then, by $A(G')x=\rho' x$, we have
\begin{eqnarray*}
\rho' x_{v_{r+1}} =x_{v_1}+\sum_{2\leq i\leq r}x_{v_i}+(n-\delta-r-2)x_{v_{r+1}}\geq x_{v_1}+(n-\delta-3)x_{v_{r+1}}.
\end{eqnarray*}
Note that $G'\in\mathcal{G}_{n,\delta+1}^{k-1}$. Then $\rho'> n-\delta-2$ by Lemma \ref{lem::2.3}, and hence
\begin{eqnarray}
 x_{v_{r+1}}\geq\frac{x_{v_1}}{\rho'-(n-\delta-3)}. \label{equ::2}
\end{eqnarray}
Assume that $E_1=\{u_1v_{i}|~r+1\leq i\leq k-1\}$ and $E_2=\{u_{i}v_j\in G' |~ 2\leq i\leq s, 1\leq j\leq t\}$. Let $G''=G'-E_2+E_1$. Clearly, $G''\cong B_{n,\delta+1}^{k-1}$. Let $y$ be the Perron vector of $A(G'')$, and let $\rho''=\rho(G'')$. By $A(G'')y=\rho'' y$, we have
\begin{eqnarray*}
 &\rho''y_{u_2}=y_{u_1}+(\delta-1)y_{u_2}.
\end{eqnarray*}
Note that $G''\in\mathcal{G}_{n,\delta+1}^{k-1}$ and $n\geq2\delta+3$. Again by Lemma \ref{lem::2.3}, $\rho''> n-\delta-2> \delta-1$, and hence
\begin{eqnarray}
 y_{u_2}=\frac{y_{u_1}}{\rho''-\delta+1}. \label{equ::3}
\end{eqnarray}
Since $G',G''\in \mathcal{G}_{n,\delta+1}^{k-1}$, by Lemma \ref{lem::2.3}, $\rho''\!-\!\rho'>-1$. Note that $x_{u_1}\geq x_{u_i}$ for $2\leq i\leq s$, $x_{v_1}\geq x_{v_j}$ for $2\leq j\leq t$. Combining this with (\ref{equ::2}), (\ref{equ::3}) and $\rho''\!-\!\rho'>-1$, we have

\begin{equation*}
\begin{aligned}
y^{T}(\rho''-\rho')x &=y^{T}(A(G'')-A(G'))x\\
   &=\sum_{u_1v_i\in E_1}(x_{u_{1}}y_{v_{i}}\!+\!x_{v_{i}}y_{u_{1}})\!
-\sum_{u_iv_j\in E_2}\!(x_{u_{i}}y_{v_{j}}\!+\!x_{v_{j}}y_{u_{i}})\\
   &\geq(k\!-\!1-r)(x_{u_1}y_{v_1}+x_{v_{r+1}}y_{u_1}-x_{u_1}y_{v_1}-x_{v_1}y_{u_2}) ~~(\mbox{since $r\leq k-2$})\\
   &=(k\!-\!1-r)(x_{v_{r+1}}y_{u_1}-x_{v_1}y_{u_2})\\
   &\geq(k\!-\!1-r)x_{v_{1}}y_{u_1}\left(\frac{1}{\rho'-(n-\delta-3)}-\frac{1}{\rho''-\delta+1}\right)
   ~~(\mbox{by (\ref{equ::2}) and (\ref{equ::3})})\\
   &=\frac{(k\!-\!1-r)x_{v_{1}}y_{u_1}}{(\rho'-(n-\delta-3))(\rho''-\delta+1)}(\rho''-\rho'+n-2\delta-2)\\
    &>0~~(\mbox{since $k\geq r\!+\!2$, $\rho'>n\!-\!\delta\!-\!2$, $\rho''>\delta\!-\!1$ and $n\geq 2\delta\!+\!3$}),
\end{aligned}
\end{equation*}
and hence $\rho''>\rho'$, which contradicts the maximality of $\rho'$. This completes the proof.
\end{proof}
For $X,Y\subseteq V(G)$, we denote by $E_{G}(X,Y)$ the set of edges with one endpoint in $X$ and one endpoint in $Y$, and $e_{G}(X,Y)=|E_{G}(X,Y)|$.
\begin{lem}\label{lem::2.6}
Let $k\geq 2$, and let $G\in \mathcal{G}_{n,a}^{k-1}$ where $n\geq 2a$, $a\geq \delta+2$ and $\delta\geq 2k$. Then $\rho(G)<\rho(B_{n,\delta+1}^{k-1}).$
\end{lem}
\begin{proof}
Let $G\in \mathcal{G}_{n,a}^{k-1}$. Then $V(G)=V(K_a)\cup V(K_{n-a})$ and $e_{G}(V(K_a),V(K_{n-a}))=k-1$. We partition $V(K_a)$ into $A_1\cup A_{2}$ and $V(K_{n-a})$ into $B_1\cup B_{2}$ such that $|A_1|=|B_1|=\delta+1$, $e_{G}(A_2, V(K_{n-a}))=0$ and $e_{G}(B_2, V(K_{a}))=0$. Since $n\geq 2a$, $a\geq \delta+2$ and $\delta\geq 2k$, it follows that $|A_2|=|V(K_a)|-|A_1|\geq 1$ and $|B_2|=|V(K_{n-a})|-|B_1|\geq 1$. Let $x$ be the Perron vector of $A(G)$. If $\sum_{v\in A_1}x_{v}<\sum_{v\in V(K_{n-a})}x_{v}$, then let $G'$ be a graph obtained from $G$ by deleting all edges between $A_1$ and $A_2$, and adding all possible edges between $A_2$ and $V(K_{n-a})$. Clearly, $G'\in \mathcal{G}_{n,\delta+1}^{k-1}$ and
\begin{equation*}
\begin{aligned}
\rho(G')-\rho(G)&\geq x^{T}(A(G')-A(G))x\\
                            &=2\sum_{v\in A_2}x_{v}\left(\sum_{v\in V(K_{n-a})}x_{v}-\sum_{v\in A_1}x_{v}\right)\\
                            &>0,
\end{aligned}
\end{equation*}
and hence $\rho(G')>\rho(G)$. Combining this with Lemma \ref{lem::2.5}, we have $\rho(B_{n,\delta+1}^{k-1})\geq\rho(G')>\rho(G)$, as required. If $\sum_{v\in A_1}x_{v}\geq\sum_{v\in V(K_{n-a})}x_{v}$, since $\sum_{v\in A_2}x_{v}>0$ and $\sum_{v\in B_2}x_{v}>0$, we have
$$\sum_{v\in V(K_{a})}x_{v}>\sum_{v\in A_1}x_{v}\geq\sum_{v\in V(K_{n-a})}x_{v}>\sum_{v\in B_1}x_{v}.$$
Let $G''$ be a graph obtained from $G$ by deleting all edges between $B_1$ and $B_2$, and adding all possible edges between $V(K_{a})$ and $B_2$. One can verify that $G''\in \mathcal{G}_{n,\delta+1}^{k-1}$ and
\begin{equation*}
\begin{aligned}
\rho(G'')-\rho(G)&\geq x^{T}(A(G'')-A(G))x\\
                            &=2\sum_{v\in B_2}x_{v}\left(\sum_{v\in V(K_a)}x_{v}-\sum_{v\in B_1}x_{v}\right)\\
                            &>0,
\end{aligned}
\end{equation*}
and hence $\rho(G'')>\rho(G)$. Similarly, $\rho(B_{n,\delta+1}^{k-1}) \geq\rho(G'') >\rho(G)$. This completes the proof.
\end{proof}

\begin{lem}\label{lem::2.7}
Let $a$ and $b$ be two positive integers. If $a\geq b$, then
$${a\choose 2}+{b\choose 2}< {a+1\choose 2}+{b-1\choose 2}.$$
\end{lem}
\begin{proof}
Note that $a\geq b$. Then
\begin{equation*}
\begin{aligned}
&{a+1\choose 2}+{b-1\choose 2}\!-\!{a\choose 2}\!-\!{b\choose 2}=a-b+1>0.
\end{aligned}
\end{equation*}
Thus the result follows.\end{proof}

Denote by $\partial_{G}(X)=E_{G}(X,V(G)-X)$.

\begin{lem}[Lemma~2.8 of \cite{Gu-Lai}]\label{lem::2.8}
Let $G$ be a graph with minimum degree $\delta$ and $U$ be a non-empty proper subset of $V(G)$. If $|\partial_{G}(U)|\leq \delta-1$, then $|U|\geq \delta+1$.
\end{lem}

For any partition $\pi$ of $V(G)$, let $E_{G}(\pi)$  denote the set of edges in $G$ whose endpoints lie in different parts of $\pi$, and let $e_{G}(\pi)=|E_{G}(\pi)|$. A part is \textit{trivial} if it contains a single vertex. The following fundamental theorem on spanning tree packing number of a graph was established by Nash-Williams \cite{Nash-Williams} and Tutte \cite{Tutte}, independently.
\begin{thm}[Nash-williams~\cite{Nash-Williams} and Tutte~\cite{Tutte}]
\label{lem::2.9}
Let $G$ be a connected graph. Then $\tau(G)\geq k$ if and only if for any partition $\pi$ of $V(G)$, $$e_{G}(\pi)\geq k(t-1),$$
where $t$ is the number of parts in the partition $\pi$.
\end{thm}

For $X\subseteq V(G)$, let $G[X]$ be the subgraph of $G$ induced by $X$, and let $e_{G}(X)$ be the number of edges in $G[X]$.
Now, we shall give the proofs of Theorems \ref{thm::edgenumber} and \ref{thm::1.2}.
\begin{proof}[\bf Proof of Theorem \ref{thm::edgenumber}]
Assume to the contrary that $\tau(G)\leq k-1$. By Theorem \ref{lem::2.9}, there exists a partition $\pi$ of $V(G)$ with $t_1$ trivial parts $v_1,v_2,\ldots, v_{t_1}$ and $t_2$ nontrivial parts $V_1,V_2,\ldots,V_{t_2}$ such that
\begin{equation}\label{equ::4}
\begin{aligned}
e_{G}(\pi)\leq k(t-1)-1,
\end{aligned}
\end{equation}
where $t=t_1+t_2$. If the partition $\pi$ contains only one part, then $e_{G}(\pi)\leq -1$ by (\ref{equ::4}), which is impossible because $e_{G}(\pi)=0$. This implies that $t\geq 2$. We assert that $t_2\geq 2$. If not, suppose that $t_2\leq 1$. Then $t_1\geq t-1$. Note that $d_{G}(v_i)\geq \delta$ for $1\leq i\leq t_1$. Combining this with $\delta\geq 2k$, we have
\begin{equation*}
\begin{aligned}
e_{G}(\pi)\geq \frac{1}{2}\sum_{1\leq i\leq t_1}d_{G}(v_i)\geq \frac{1}{2}\delta t_1\geq k(t-1),
\end{aligned}
\end{equation*}
which contradicts (\ref{equ::4}). This implies that $t_2\geq 2$. If the partition $\pi$ contains at most one nontrivial part, say $V_j$ ($1\leq j\leq t_2$), such that  $|\partial_{G}(V_j)|\leq \delta-1$. Then $|\partial_{G}(V_i)|\geq \delta$ for all $i\in\{1,\ldots,t_2\}\backslash\{j\}$. Since $G$ is connected, it follows that $|\partial_{G}(V_j)|\geq 1$, and hence
\begin{equation*}
\begin{aligned}
2e_{G}(\pi)&= \sum_{1\leq i\leq t_2}|\partial_{G}(V_i)|+\sum_{1\leq j\leq t_1}d_{G}(v_j)\\
&\geq (t_2-1)\delta+1+\delta t_1\\
&= \delta(t-1)+1~~(\mbox{since $t=t_1+t_2$})\\
&\geq 2k(t-1)+1 ~~(\mbox{since $\delta\geq 2k$}),
\end{aligned}
\end{equation*}
which also contradicts (\ref{equ::4}). Therefore, the partition $\pi$ contains at least two nontrivial parts, say $V_1,V_2$, such that $|\partial_{G}(V_i)|\leq \delta-1$ for $i=1,2$. Furthermore, by Lemma~\ref{lem::2.8}, we obtain $|V_i|\geq \delta+1$ for $i=1,2$. If $|V_1|=\max\{|V_1|,|V_2|,\ldots,|V_{t_2}|\}$ or $|V_2|=\max\{|V_1|,|V_{2}|,\ldots,|V_{t_2}|\}$, since $|V_{i}|\geq\delta+1$ and $|V_{j}|\geq2$ for $i=1,2$ and $3\leq j\leq t_2$, by Lemma \ref{lem::2.7},
\begin{equation*}
\begin{aligned}
\sum_{1\leq i\leq t_2}e_{G}(V_i)&\leq {\delta+1\choose 2}+(t_2-2){2\choose 2}+{n-(t+\delta+t_2-3)\choose 2}\\
&\leq{\delta+1\choose 2} +{n-(t+\delta-1)\choose 2} ~~(\mbox{since $t_2\geq 2$}).
\end{aligned}
\end{equation*}
If there exists a nontrivial part, say $V_j$, such that  $|V_j|=\max\{|V_1|,|V_2|,\ldots,|V_{t_2}|\}$ for some $3\leq j\leq t_2$. Similarly,
$$\sum_{1\leq i\leq t_2}e_{G}(V_i)\leq 2{\delta+1\choose 2} +{n-(2\delta+t-1)\choose 2}.$$
Note that $|V_1|\geq \delta+1$ and $|V_2|\geq \delta+1$. Then $2\leq t\leq n-2\delta$. Combining this with (\ref{equ::4}) and $\sum_{1\leq i\leq t_1}e_{G}(v_i)=0$, we have
\begin{equation*}
\begin{aligned}
e(G)&=\sum_{1\leq i\leq t_2}e_{G}(V_i)+\sum_{1\leq i\leq t_1}e_{G}(v_i)+e_{G}(\pi)\\
&\leq \max\left\{{\delta\!+\!1\choose 2}\!+\!{n\!-\!(t\!+\!\delta\!-\!1)\choose 2},2{\delta\!+\!1\choose 2}\!+\!{n\!-\!(2\delta\!+\!t\!-\!1)\choose 2}\right\}\!+\!k(t\!-\!1)\!-\!1\\
&\leq{\delta\!+\!1\choose 2}\!+\!{n\!-\!(t\!+\!\delta\!-\!1)\choose 2}+k(t\!-\!1)\!-\!1~~(\mbox{since $t\leq n-2\delta$})\\
&=\frac{t^2}{2}+(-n+\delta-\frac{1}{2}+k) t+\delta^2-n\delta+ \frac{n^2}{2}+\frac{n}{2}-k-1.
\end{aligned}
\end{equation*}
Let $g(t)=\frac{t^2}{2}+(-n+\delta-\frac{1}{2}+k) t+\delta^2-n\delta+ \frac{n^2}{2}+\frac{n}{2}-k-1$. We take the derivative of $g(t)$. Thus
$$g'(t)=t+\delta+k-n-\frac{1}{2}\leq k-\delta-\frac{1}{2}<0$$
by the facts that $\delta\geq 2k$, $k\geq 1$ and $t\leq n-2\delta$.
This implies that $g(t)$ is decreasing with respect to $2\leq t\leq n-2\delta$, and hence
$$e(G)\leq {\delta\!+\!1\choose2}\!+\!{n\!-\!\delta\!-\!1\choose 2}+k-1,$$
a contradiction. This completes the proof.\end{proof}

\begin{proof}[\bf Proof of Theorem \ref{thm::1.2}]
Assume to the contrary that $\tau(G)\leq k-1$. By Theorem~\ref{lem::2.9}, there exists a partition $\pi$ of $V(G)$ with $t_1$ trivial parts $v_1,v_2,\ldots, v_{t_1}$ and $t_2$ nontrivial parts $V_1,V_2,\ldots,V_{t_2}$ such that
\begin{equation}\label{equ::5}
\begin{aligned}
e_{G}(\pi)\leq k(t-1)-1,
\end{aligned}
\end{equation}
where $t=t_1+t_2$. By using the same analysis as Theorem \ref{thm::edgenumber}, we can deduce that the partition $\pi$ contains at least two nontrivial parts, say $V_1,V_2$, such
that $|V_i|\geq\delta+1$ for $i=1,2$.
First suppose that $t=2$. This implies that the partition $\pi$ consists of two nontrivial parts $V_1,V_2$.  Then $e_{G}(\pi)=e_{G}(V_1,V_2)\leq k-1$ by (\ref{equ::5}).  Clearly, $G$ is a spanning subgraph of some graph $H$ in $\mathcal{G}_{n,|V_1|}^{k-1}$. Then
\begin{equation}\label{equ::6}
\rho(G)\leq \rho(H),
\end{equation}
with equality if and only if $G\cong H$.
Note that $\min\{|V_1|,|V_2|\}\geq\delta+1$. Combining this with Lemmas~\ref{lem::2.5} and \ref{lem::2.6} as well as (\ref{equ::6}), we conclude that $$\rho(G)\leq\rho(B_{n,\delta+1}^{k-1}),$$
with equality if and only if $G\cong B_{n,\delta+1}^{k-1}$. However, this is impossible because $\rho(G)\geq\rho(B_{n,\delta+1}^{k-1})$ and $G\ncong B_{n,\delta+1}^{k-1}$.

Now suppose that $t\geq 3$. Note that $\rho(G)\geq\rho(B_{n,\delta+1}^{k-1})>\rho(K_{n-\delta-1})=n-\delta-2$. Then by Lemmas \ref{lem::2.1} and \ref{lem::2.2},
\begin{equation}\label{equ::7}
e(G)> \frac{n^2}{2}-\frac{(2\delta+3)n}{2}+(\delta+1)^{2}.
\end{equation}
Since $\min\{|V_1|,|V_2|\}\geq\delta+1$, by using a similar analysis as Theorem \ref{thm::edgenumber}, it follows that
\begin{equation*}
\begin{aligned}
e(G)&=\sum_{1\leq i\leq t_2}e_{G}(V_i)+\sum_{1\leq i\leq t_1}e_{G}(v_i)+e_{G}(\pi)\\
&\leq \max\left\{{\delta\!+\!1\choose 2}\!+\!{n\!-\!(t\!+\!\delta\!-\!1)\choose 2},2{\delta\!+\!1\choose 2}\!+\!{n\!-\!(2\delta\!+\!t\!-\!1)\choose 2}\right\}\!+\!k(t\!-\!1)\!-\!1\\
&\leq{\delta\!+\!1\choose 2}\!+\!{n\!-\!(t\!+\!\delta\!-\!1)\choose 2}+k(t\!-\!1)\!-\!1~~(\mbox{since $t\leq n-2\delta$})\\
&= \frac{n^2}{2}-\frac{(2t+2\delta-1)n}{2}+\frac{(t+2\delta+2k-1)t}{2}+\delta^2-k-1.
\end{aligned}
\end{equation*}
Combining this with (\ref{equ::7}) and $t\geq 3$, we have
\begin{equation}\label{equ::8}
n< \delta+k+\frac{t+1}{2}+\frac{k-1}{t-2}.
\end{equation}
Suppose that $f(t)=\delta+k+\frac{t+1}{2}+\frac{k-1}{t-2}$. One can verify that $f(t)$ is convex for $t>0$ and its maximum in any closed interval is attained at one of the ends of this interval. Note that $3\leq t\leq n-2\delta$. Then
\begin{equation*}
\begin{aligned}
f(t)&\leq \max\left\{\delta+2k+1, \frac{n+1}{2}+k+\frac{k-1}{n-2\delta-2}\right\}
\leq \frac{n-1}{2}+2k~(\mbox{since $n\geq 2\delta +3$}).
\end{aligned}
\end{equation*}
Combining this with (\ref{equ::8}) and $\delta\geq 2k$, we can deduce that $n<4k-1\leq 2\delta -1$, a contradiction. This completes the proof.
\end{proof}

\section{Proof of Theorem  \ref{thm::1.3}}\label{sec::econ}
The proof idea of Theorem~\ref{thm::1.3} is quite similar to that of Theorem~\ref{thm::1.2}. To present the proof, we need several lemmas below.

\begin{lem}[Theorem~3.1 of \cite{Ning}]\label{lem::3.1}
If $G_0$ is a graph with the maximum spactral radius in $\mathcal{A}_{n}^{\kappa',\delta}$, where $1\leq\kappa'<\delta$, then $G_0\in \mathcal{G}_{n,\delta+1}^{\kappa'}$.
\end{lem}

Recall that $\mathcal{G}_{n,n_1}^{i}$ is the set of graphs obtained from $K_{n_1}\cup K_{n-n_{1}}$ by adding $i$ edges between $K_{n_1}$ and $K_{n-n_{1}}$, and $B_{n,\delta+1}^{i}$ is the graph obtained from $K_{\delta+1}\cup K_{n-\delta-1}$ by adding $i$ edges joining a vertex in $K_{\delta+1}$ and $i$ vertices in $K_{n-\delta-1}$.
We can extend the results of Lemmas \ref{lem::2.3} and \ref{lem::2.5}.
\begin{lem}\label{lem::3.2}
Let $G\in\mathcal{G}_{n,\delta+1}^{\kappa'}$, where $n\geq 2\delta+4$ and $4\leq \kappa'< \delta$. Then
$$n-\delta-2< \rho(G)< n-\delta.$$
\end{lem}
\begin{proof}
Note that $G$ contains $K_{\delta+1}\cup K_{n-\delta-1}$ as a proper spanning subgraph and $n\geq 2\delta+4$. Then  $\rho(G)>\rho(K_{\delta+1}\cup K_{n-\delta-1})= n-\delta-2$.
Since $G\in\mathcal{G}_{n,\delta+1}^{\kappa'}$ and $\delta> \kappa'$, we have
\begin{equation*}
\begin{aligned}
   e(G)&={\delta+1\choose 2}+{n-\delta-1\choose 2}+\kappa'\\
   &\leq {\delta+1\choose 2}+{n-\delta-1\choose 2}+\delta-1\\ &=\frac{n^2}{2}-\frac{(2\delta+3)n}{2}+\delta^2+3\delta.
\end{aligned}
\end{equation*}
Combining this with Lemmas \ref{lem::2.1} and \ref{lem::2.2}, we have
\begin{equation*}
\begin{aligned}
   \rho(G)&\leq \frac{\delta-1}{2}+\sqrt{n^2 + (-3\delta - 3)n + \frac{9\delta^2}{4}+ \frac{13\delta}{2}+\frac{1}{4}}\\
   &= \frac{\delta-1}{2}+\sqrt{\left(n-\frac{3\delta}{2}+\frac{1}{2}\right)^2-4(n-2\delta )}\\
   &< \frac{\delta-1}{2}+ \left(n-\frac{3\delta}{2}+\frac{1}{2}\right)~~(\mbox{since $n\geq 2\delta\!+\!4$})\\
   &=n-\delta.
\end{aligned}
\end{equation*}
Thus $n-\delta-2< \rho(G)< n-\delta$, as required.\end{proof}

By Lemma \ref{lem::3.2} and using the same analysis as the proof of Lemma \ref{lem::2.5}, we easily obtain the following result.
\begin{lem}\label{lem::3.3}
Let $G\in \mathcal{G}_{n,\delta+1}^{\kappa'}$ where $4\leq \kappa'< \delta$ and $n\geq 2\delta+4$. Then $\rho(G)\leq \rho(B_{n,\delta+1}^{\kappa'})$, with equality if and only if $G\cong B_{n,\delta+1}^{\kappa'}$.
\end{lem}

\begin{proof}[\bf Proof of Theorem \ref{thm::1.3}]
The result follows from Lemmas~\ref{lem::3.1} and \ref{lem::3.3}.
\end{proof}

\section{Some applications}\label{sec::app}

Edge-disjoint spanning trees are closely related to many graph properties, such as collapsibility, supereulerianity, spanning connectivity, nowhere-zero flows, group connectivity, rigidity, and others \cite{CDGG22,LL19,Palmer01}. We introduce several applications of our main result in spectral graph theory.

Spectral conditions of classic rigidity have been studied in \cite{CDGG22,CDG21,FHL22}. We present spectral conditions for two other variations of rigidity in the next two subsections. To conclude, we also present applications in nowhere-zero flows at the end of this section.

\subsection{Body-and-bar rigidity}
A \textit{body-and-bar frameworks} in $\mathbb{R}^d$ is a framework of $d$-dimension rigid bodies that are connected by fixed-length bars attached at points of their surfaces (see \cite{Tay84} for more details). Informally, we say a graph $G$ is \textit{body-bar rigid in $\mathbb{R}^d$} if there exists a generic rigid body-bar framework in $\mathbb{R}^d$. 
Instead of a formal definition, we present the following characterization.

\begin{thm}[Tay~\cite{Tay84}]
A graph $G$ is body-bar rigid in $\mathbb{R}^d$ if and only if it contains $\frac{d(d+1)}{2}$ edge-disjoint spanning trees.
\end{thm}

By the above theorem and Theorem~\ref{thm::1.2}, we have the following spectral condition for body-bar rigidity.
\begin{pro}
Let $k=\frac{d(d+1)}{2}$, and let $G$ be a connected graph with minimum degree $\delta\geq 2k$ and order $n\geq 2\delta+3$. If $\rho(G)\geq \rho(B_{n,\delta+1}^{k-1})$, then $G$ is body-bar rigid in $\mathbb{R}^d$ unless $G\cong B_{n,\delta+1}^{k-1}$.
\end{pro}

\subsection{Rigidity on surfaces of revolution}
Here we assume that the joints of our framework are restricted to lie on a smooth surface $\mathcal{M} \subset \mathbb{R}^3$, and $\mathcal{M}$ is an \textit{irreducible surface}; i.e.,~$\mathcal{M}$ is the zero set of an irreducible rational polynomial $h(x,y,z) \in \mathbb{Q}[X,Y,Z]$. The framework $(G,p)$ with $p(v) \in \mathcal{M}$ for every $v \in V(G)$ is \textit{rigid on $\mathcal{M}$} if there exists  $\varepsilon >0$ such that if $(G,p)$ is equivalent to $(G,q)$ and $\|p(v)-q(v)\|<\epsilon$ and $q(v) \in \mathcal{M}$ for every $v \in V(G)$, then $(G,p)$ is congruent to $(G,q)$.
An irreducible surface is called an \textit{irreducible surface of revolution} if it can be generated by rotating a continuous curve about a fixed axis. See \cite{NixonOwenPower12} for more information.

\begin{thm}[Nixon, Owen and Power \cite{NixonOwenPower12,NixonOwenPower14}]
\label{thm::nop}
Let $\mathcal{M}$ be an irreducible surface of revolution.
Then a graph $G$ is rigid on $\mathcal{M}$ if and only if one of the following:
\\$(i)$ $G$ is a complete graph,
\\$(ii)$ $\mathcal{M}$ is a sphere and $G$ contains a spanning Laman graph,
\\$(iii)$ $\mathcal{M}$ is a cylinder and $G$ contains two edge-disjoint spanning trees, or
\\$(iv)$ $\mathcal{M}$ is not a cylinder or a sphere and $G$ contains two edge-disjoint spanning subgraphs $G_1,G_2$, where $G_1$ is a tree and every connected component of $G_2$ contains exactly one cycle.
\end{thm}

By the above theorem, the results in \cite{FHL22} actually imply spectral conditions for the rigidity on sphere. By Theorem~\ref{thm::1.2}, we have the following spectral condition for rigidity on irreducible surfaces of revolution that is not a sphere.
\begin{pro}
Let $G$ be a connected graph with minimum degree $\delta\geq 4$ and order $n\geq 2\delta+3$. If $\rho(G)\geq \rho(B_{n,\delta+1}^{1})$, then $G$ is rigid on any irreducible surface of revolution that is not a sphere unless $G\cong B_{n,\delta+1}^{1}$.
\end{pro}
\begin{proof}
By Theorem~\ref{thm::1.2}, $G$ contains two edge-disjoint spanning trees. Since $\delta\ge 4$, we have $|E(G)|\ge n\delta/2\ge 2n$. Thus $G$ has extra edges that are not in the two edge-disjoint spanning trees. It follows that $G$ contains a spanning tree and a spanning subgraph with exactly one cycle that are edge-disjoint. By Theorem~\ref{thm::nop}, the result follows.
\end{proof}

\subsection{Nowhere-zero flows}
The theory of integer flows was initiated by Tutte. For an orientation $D$ of a graph $G$ and for each vertex $v$, let $E_D^+(v)$ and $E_D^-(v)$ be the set of edges oriented away from $v$ and the set of edges oriented into $v$, respectively. A \textit{nowhere-zero $k$-flow} of a graph $G$ is an orientation $D$ together with a function $f: E(G)\rightarrow \{\pm 1, \pm 2,\ldots, \pm(k-1)\}$ such that
$\sum_{e\in E_D^+(v)} f(e) = \sum_{e\in E_D^-(v)} f(e)$ for each vertex $v$. As a generalization of nowhere-zero $k$-flow, a \textit{nowhere-zero circular $k/d$-flow}, introduced by Goddyn, Tarsi and Zhang~\cite{GTZ98}, is a nowhere-zero $k$-flow such that the range of $f$ is contained in $\{\pm d, \pm (d+1),\ldots, \pm(k-d)\}$. The \textit{flow index $\phi(G)$} of a graph $G$ is the least rational number $r$ such that $G$ admits a nowhere-zero circular $r$-flow. The central problems in this research area are the three well-known flow conjectures proposed by Tutte.

For the application of spanning tree packing number in nowhere-zero flow theory, we refer readers to~\cite{LL19}. In particular, it is proved in~\cite{LXZ07} that if $\tau(G)\ge 3$, then $\phi(G)<4$; and in~\cite{HLL18} that if $\tau(G)\ge 4$, then $\phi(G)\le 3$. Thus, by Theorem~\ref{thm::1.2}, we can easily obtain sufficient conditions of flow index via spectral radius, however, these spectral conditions might not be best possible. It is worth studying flow index directly from spectral perspectives in the future.

\section{Concluding remarks}\label{sec::remarks}

Theorem~\ref{thm::1.2} actually implies that $B_{n,\delta+1}^{\tau}$ is the unique graph that has the maximum spectral radius among all graphs of fixed order $n$ with minimum degree $\delta$ and spanning tree packing number $\tau$. 
Let $G$ be a minimum graph with $\tau(G)\ge k$ and of order $n$, that is, $G$ consists of exactly $k$ edge-disjoint spanning trees with no extra edges. This implies that $\tau(G)=k$ and $e(G)=k(n-1)$. We are interested in the maximum possible spectral radius of $G$.

\begin{prob}\label{prob::minkst}
Let $G$ be a minimum graph with $\tau(G)\ge k$ and of order $n$. For each $n\ge 4$ and each $k\ge 2$, determine the maximum possible spectral radius of $G$ and characterize extremal graphs.
\end{prob}

To attack Problem~\ref{prob::minkst}, notice that $\delta(G)\ge k$ and $e(G)=k(n-1)$, we can easily obtain an upper bound on $\rho(G)$ by Lemma~\ref{lem::2.1}. However, this upper bound may not be tight, since for most values of $n$, $G$ cannot be $k$-regular or a bidegreed graph in which each vertex is of degree either $k$ or $n-1$. To attain the maximum spectral radius, it seems that $G$ is obtained from a $K_{2k}$ by continuously adding a vertex and $k$ incident edges step by step until $G$ has $n$ vertices, and in particular, we guess the graph is $K_{k}\nabla (K_{k}\cup (n-2k)K_1)$, where $\nabla$ and $\cup$ denote the join and the union of two graphs, respectively. We leave this as an open question.

Nash-Williams~\cite{Nash64} ever studied the forest covering problem, seeking the minimum number of forests that cover the entire graph. This is like a dual problem of spanning tree packing. The {\it arboricity} $a(G)$ is the minimum number of edge-disjoint forests whose union equals $E(G)$.
\begin{thm}[Nash-Williams~\cite{Nash64}]\label{thm::5.1}
Let $G$ be a connected graph. Then $a(G)\le k$ if and only if for any subgraph $H$ of $G$, $|E(H)|\le k(|V(H)|-1)$.
\end{thm}

Naturally, we have the following problem.
\begin{prob}\label{prob::arbo}
Find a tight spectral radius condition for a graph $G$ of order $n$ with $a(G)\le k$ and characterize extremal graphs.
\end{prob}

When $n\leq 2k$, Problem~\ref{prob::arbo} is trivial. Any graph $G$ of order $n\leq 2k$ has the property $a(G)\le k$ and so the extremal graph is $K_n$. To see this, for any subgraph $H$ of $G$, we can deduce that $|E(H)|\leq {|V(H)|\choose 2} \leq k(|V(H)|-1)$ when $n\leq 2k$, and the conclusion follows from Theorem~\ref{thm::5.1}. The situation becomes more involved for $n\geq 2k+1$. In fact, this case is even stronger than Problem~\ref{prob::minkst}, since if $G$ consists of exactly $k$ edge-disjoint spanning trees with no extra edges, we have $a(G) =\tau(G)=k$.
Notice that $a(K_{k}\nabla (K_{k}\cup (n-2k)K_1))=k$ and $e(K_{k}\nabla (K_{k}\cup (n-2k)K_1))=k(n-1)$, 
thus we guess that the extremal graph w.r.t. the spectral radius is also $K_{k}\nabla (K_{k}\cup (n-2k)K_1)$. This is left for possible future work.


\section*{Acknowledgements}
Xiaofeng Gu was supported by a grant from the Simons Foundation (522728), and Huiqiu Lin was supported by the National Natural Science Foundation of China (No. 12011530064) and Natural Science Foundation of Shanghai (No. 22ZR1416300).

\end{document}